\documentclass[12pt,reqno,a4paper]{amsart}
\usepackage{amsmath,amssymb,amsthm,amsaddr}
\usepackage{enumitem}
\usepackage[margin=2.2cm,top=3.5cm,bottom=2.9cm,footskip=1cm,headsep=1cm]{geometry}
\usepackage[utf8]{inputenc}
\usepackage{graphicx}
\usepackage{tikz}
\usepackage{mathtools}
\usetikzlibrary{arrows}

\author{H. Egger}
\address{Department of Mathematics, TU Darmstadt, Germany}
\email{egger@mathematik.tu-darmstadt.de}

\title[A mixed finite element method for compressible flow]{A robust conservative mixed finite element method\\ for compressible flow on pipe networks}

\newtheorem{lemma}{Lemma}[section]
\newtheorem{problem}[lemma]{Problem}

\theoremstyle{definition}
\newtheorem{remark}[lemma]{Remark}
\newtheorem{notation}[lemma]{Notation}
\newtheorem*{example*}{Example}

\def\div{\mathrm{div}}

\def\dt{\partial_t}
\def\dtau{\bar\partial_\tau}
\def\dx{\partial_x}

\def\dxx{\partial_{xx}}

\def\RR{\mathbb{R}}

\def\E{\mathcal{E}}

\def\G{\mathcal{G}}

\def\V{\mathcal{V}}
\def\Vi{{\V_{0}}}
\def\Vb{{\V_{\partial}}}

\numberwithin{equation}{section}
\numberwithin{table}{section}
\numberwithin{figure}{section}

\begin{document}

\begin{abstract} 
We consider the numerical approximation of compressible flow in a pipe network. Appropriate coupling conditions are formulated that allow us to derive a variational characterization of solutions and to prove global balance laws for the conservation of mass and energy on the whole network. This variational principle, which is the basis of our further investigations, is amenable to a conforming Galerkin approximation by mixed finite elements. The resulting semi-discrete problems are well-posed and automatically inherit the global conservation laws for mass and energy from the continuous level. We also consider the subsequent discretization in time by a problem adapted implicit time stepping scheme which leads to conservation of mass and a slight dissipation of energy of the full discretization. The well-posedness of the fully discrete scheme is established and a fixed-point iteration is proposed for the solution of the nonlinear systems arising in every single time step. Some computational results are presented for illustration of our theoretical findings and for demonstration of the robustness and accuracy of the new method.
\end{abstract}

\maketitle

\vspace*{-1em}

\begin{quote}
\noindent 
{\small {\bf Keywords:} 
compressible flow, 
variational methods,
energy estimates,
Galerkin approximation,
mixed finite elements,
implicit time discretization
}
\end{quote}

\begin{quote}
\noindent
{\small {\bf AMS-classification (2000):}
35D30,35R02,37L65,76M10,76N99}
\end{quote}

\section{Introduction} \label{sec:intro}

This paper addresses the numerical approximation of flow problems governing the propagation of a compressible fluid in a pipeline network.
On every single pipe $e$, the conservation of mass and the balance of momentum 
shall be modeled by 
\begin{align*}
\dt \rho_e + \dx m_e &= 0,\\
\dt m_e + \dx \Big(\frac{m_e^2}{\rho_e} + p_e\Big) &= a_e \rho_e \dx( \frac{1}{\rho_e^2} \dx m_e) - b_e \frac{|m_e| m_e }{\rho_e}. 
\end{align*}
Here $\rho_e$ is the density, $m_e=\rho_e u_e$ is the mass flux, and $u_e$ is the velocity of the flow. The parameters $a_e,b_e$ are assumed to be constant and non-negative.
As equation of state, relating the density $\rho_e$ to the pressure $p_e$, we utilize 
\begin{align*} 
p_e(\rho) = c_e \rho^\gamma, \qquad \gamma>1, \ c_e > 0,
\end{align*}
and we use a corresponding potential energy density of the form \cite{NovotnyStraskraba04}
\begin{align*}
P_e(\rho) = \rho \int_0^\rho p_e(s) / s^2 ds = \frac{c_e}{\gamma-1} \rho^\gamma. 
\end{align*}
As indicated below, more general situations including the isothermal flow of real gas scan be considered as well.
The above equations describe the conservation of mass and energy in isentropic flow of a compressible fluid within the pipe $e$ with possible dissipation of energy through viscous forces and friction at the pipe walls. 
The particular form of the viscous term will become clear from our analysis.
Note that for constant density $\rho = \bar \rho$ we have
\begin{align*}
\bar \rho \dx \left(\frac{1}{\bar \rho^2} \dx m\right) = \dxx u.
\end{align*}
The viscous term in the above equations thus is in principle of the form as the one usually employed in the compressible Navier-Stokes equations \cite{Feireisl03,Lions98,NovotnyStraskraba04}.
In order to correctly describe the conservation of mass and energy at junctions $v$ of several pipes $e \in \E(v)$ of the network, we require the following coupling conditions to hold true at the junctions:
\begin{align*}
\sum_{e \in \E(v)} m_e(v) n_e(v) &= 0,\\
\frac{m_e(v)^2}{2 \rho_e(v)^2}+P'_e(\rho_e(v))  &= h_v +  \frac{a_e}{\rho_e^2}\dx m_e(v), \qquad \forall e \in \E(v).
\end{align*}
Note that the value of $h_v$ is assumed to be independent of $e$. 
For the inviscid case $a_e=0$, the second equation describes the continuity of the  \emph{specific stagnation enthalpy} \cite{MorinReigstad15,Reigstad15}; 
here we additionally take into account the viscous forces.
The two coupling conditions imply that the fluxes of mass and energy through a junction $v$ sum up to zero and therefore no mass or energy is generated or annihilated. This will become clear from our analysis below.

\medskip 

There has been intensive discussion about the appropriate coupling conditions for compressible flow in pipe networks. 
The first condition stated above is equivalent to conservation 
of mass at the junction and out of doubt; see e.g. \cite{BrouwerGasserHerty11,Garavello10,Osiadacz84}.
In contrast to that, various different relations have been considered as a second coupling condition: 
continuity of the pressure $p$ is used
frequently in the modeling and simulation of pipeline networks \cite{BandaHertyKlar06a,GugatEtAl12,Osiadacz84}. 
As shown in \cite{ColomboGaravello08,Reigstad15}, this condition may lead to 
unphysical solutions for junctions of more than two pipes. 
Also the continuity of the dynamic pressure $m^2/\rho+p(\rho)$ proposed in
\cite{ColomboGaravello06,ColomboGaravello08,ColomboHertySachers08} may in general lead to unphysical solutions; see \cite{MorinReigstad15,Reigstad15}.
The continuity of the stagnation enthalpy on the other hand leads to entropic solutions for all subsonic flow conditions for junctions connecting several pipes of arbitrary cross-sectional area \cite{Reigstad15}.
Our second coupling condition in addition accounts for the viscous forces. In the course of the manuscript, we will provide a derivation of the two coupling conditions based only on the rationale that the mass and energy should be conserved at the junctions.

\medskip

A vast amount of literature is devoted to numerical methods for compressible flow in a single pipe. Most of the schemes that have been proven to be globally convergent to weak solutions of the compressible Navier-Stokes equations are formulated in Lagrangian coordinates \cite{ZarnowskiHoff91,ZhaoHoff94,ZhaoHoff97}. This makes their generalization to networks practically infeasible. 
Finite volume methods on the other hand are typically formulated in a Eulerian framework and many partial results on their stability and convergence are available; see \cite{LeVeque02} and the references given there. The correct handling of coupling conditions in the network context, however, seems not straight forward; see \cite{BressanEtAl15} for analytical reasons. 
A globally convergent non-conforming finite element for the compressible Navier-Stokes equations has been proposed and analyzed recently \cite{GallouetEtAl16,Karper13,Karper14}. 
Due to the somewhat unusual form of the coupling conditions, a direct generalization of this method to pipe networks again seems not feasible. 

\medskip 

In this paper, we therefore propose an alternative strategy for the numerical approximation of one-dimensional compressible flow that naturally generalizes to networks and that allows to establish the conservation of mass and energy in a rather direct manner. 
The two coupling conditions stated above naturally arise in the derivation of a special variational characterization of solutions for the compressible flow problem given below. 
This variational principle encodes the conservation of mass and energy much more directly as previous formulations and in addition turns out to be amenable to a conforming Galerkin approximation in space. As a particular discretization, we consider in detail a mixed finite element method for which we prove conservation of mass and energy independently of the topology of the network and of the mesh size.
In addition, we also investigate the discretization in time by a problem adapted implicit time stepping scheme that allows to establish global balance relations for mass and energy also for the fully discrete scheme.
The final method provides exact conservation of mass and a slight dissipation of 
energy due to numerical dissipation caused by the implicit time stepping strategy. 
We establish the well-posedness of the fully discrete scheme and consider a problem adapted fixed-point iteration for the numerical solution of the nonlinear systems arising in every time step. 

\medskip

The numerical approximations obtained with our method automatically satisfy uniform energy bounds. This is the first step and main ingredient for the proof of global existence of weak solutions to the compressible Navier-Stokes equations \cite{Feireisl03,Lions98,NovotnyStraskraba04} and also for their systematic numerical approximation \cite{GallouetEtAl16,Karper13,Karper14}. 
Our method is comparably simple and directly inherits the basic conservation principle for mass and energy from the continuous problem. Moreover, our approach naturally extends to pipeline networks. 
We strongly belief that with similar arguments as in \cite{Karper14}, it might be possible to obtain a complete convergence analysis also for the method proposed in this paper.
This is however left as a topic for future research.

\medskip 

The remainder of the manuscript is organized as follows: 
In Section~\ref{sec:prelim}, we introduce our notation and provide a complete definition 
of the compressible flow problem to be considered. 
Section~\ref{sec:variational} is then devoted to the derivation of the variational principle, which is the basis for the design of our numerical method. 
In Section~\ref{sec:galerkin}, we discuss the discretization of the weak formulation in space by a conforming Galerkin approach using mixed finite elements, and we establish conservation of mass and energy for the resulting semi-discretization.
Section~\ref{sec:time} is then devoted to the subsequent discretization in time.
In Section~\ref{sec:implementation}, we discuss some details of the implementation of the fully discrete scheme. 
For illustration of our theoretical results, we also present some preliminary computational results for standard test problems in Section~\ref{sec:numerics}. 
We close with a short summary and a discussion of topics for future research.

\section{Notation and problem statement} \label{sec:prelim}

Let us briefly summarize our main notation used in the rest of the paper 
and then provide a complete definition of the flow problem under consideration.

\subsection{Topology and geometry}

The topology of the pipe network will be represented by a finite, directed, and connected graph $\G=(\V,\E)$ with vertices $\V=\{v_1,\ldots,v_n\}$ and edges $\E=\{e_1,\ldots,e_m\} \subset \V \times \V$. 
For any edge $e=(v_1,v_2)$, we denote by $\V(e)=\{v_1,v_2\}$ the set of vertices of the edge $e=(v_1,v_2)$, and we denote by $\E(v)=\{ e=(v,\cdot) \text{ or } e=(\cdot,v)\}$ the set of edges incident on $v$. 
The set of vertices can be split into inner vertices $\Vi=\{v : |\E(v)| \ge 2\}$ and boundary vertices $\Vb = \V \setminus \Vi$. 
For any edge $e=(v_1,v_2)$, we further define two numbers
\begin{align*}
n_e(v_1)=-1 \qquad \text{and} \qquad n_e(v_2)=1,
\end{align*}
marking the in- and outflow vertex and thus defining the orientation of the edge.
The matrix $N \in \RR^{n \times m}$ given by $N_{ij} = n_{e_j}(v_i)$ is usually called the \emph{incidence matrix} of the graph \cite{Berge}. 
For illustration of the above notions by a particular example, see Figure~\ref{fig:graph}.
\begin{figure}[ht!]
\begin{minipage}[c]{.3\textwidth}
\hspace*{-2.5em}
\begin{tikzpicture}[scale=.6]
\node[circle,draw,inner sep=2pt] (v1) at (0,2) {$v_1$};
\node[circle,draw,inner sep=2pt] (v2) at (4,2) {$v_2$};
\node[circle,draw,inner sep=2pt] (v3) at (8,4) {$v_3$};
\node[circle,draw,inner sep=2pt] (v4) at (8,0) {$v_4$};
\draw[->,thick,line width=1.5pt] (v1) -- node[above] {$e_1$} ++(v2);
\draw[->,thick,line width=1.5pt] (v2) -- node[above,sloped] {$e_2$} ++(v3);
\draw[->,thick,line width=1.5pt] (v2) -- node[above,sloped] {$e_3$} ++(v4);
\end{tikzpicture}
\end{minipage}
\caption{\label{fig:graph}Graph $\G=(\V,\E)$ with vertices $\V=\{v_1,v_2,v_3,v_4\}$ and edges $\E=\{e_1,e_2,e_3\}$
defined by $e_1=(v_1,v_2)$, $e_2=(v_2,v_3)$, and $e_3=(v_2,v_4)$. 
Here $\Vi=\{v_2\}$, $\Vb=\{v_1,v_3,v_4\}$, and $\E(v_2)=\{e_1,e_2,e_3\}$, 
and the non-zero entries of the incidence matrix are $n_{e_1}(v_1)=n_{e_2}(v_2)=n_{e_3}(v_2)=-1$ and $n_{e_1}(v_2)=n_{e_2}(v_3)=n_{e_3}(v_4)=1$.} 
\end{figure}
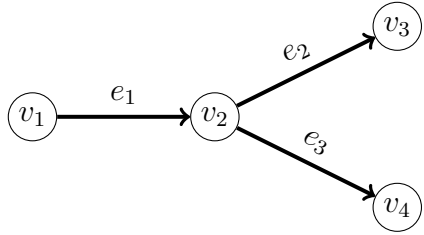

To each edge $e \in \E$, we further associate a parameter $\ell_e>0$ representing the length of the corresponding pipe.
Throughout the presentation, we tacitly identify the interval $[0,\ell_e]$ with the edge $e$ which it corresponds to. 
The values $\ell_e$ are stored in a vector $\ell=(\ell_e)_{e \in \E}$.
The triple $\G=(\V,\E,\ell)$ is called a \emph{geometric graph} \cite{Mugnolo14} 
and serves as the basic geometric model for the pipe network for the rest of the manuscript.

\subsection{Function spaces}
We denote by
\begin{align*}
L^2 = L^2(\E) = \{ u : u|_e=u_e \in L^2(e) \quad \forall e \in \E\}
\end{align*}
the space of square integrable functions defined over the network. 
Here and below, we tacitly identify $L^2(e)$ with $L^2(0,l_e)$ and with slight abuse of notation, we write
\begin{align*}
(u_e,v_e)_{e}=\int_e u_e v_e dx = \int_{v_1}^{v_2} u_e v_e dx = \int_0^{l_e} u_e v_e dx. 
\end{align*}
The scalar product of $L^2(\E)$ is then defined by
\begin{align*}
(u,v)_\E = \sum\nolimits_{e \in \E} (u_e,v_e)_{e}. 
\end{align*}
We further denote by 
\begin{align*}
H^s(\E) = \{ u \in L^2(e): u_e \in H^s(e) \quad \forall e \in \E\} 
\end{align*}
the broken Sobolev spaces of piecewise smooth functions. 
The broken derivative of a piecewise smooth function $u \in H^1(\E)$ is denoted by $\dx' u$ and given by 
\begin{align*}
(\dx' u)|_e = \dx (u|_e) \qquad \text{for all } e \in \E. 
\end{align*}
This allows us to equivalently define $H^1(\E) = \{v \in L^2(\E) : \dx' v \in L^2(\E)\}$.
Note that for any parameter $s>1/2$, functions in $H^s(\E)$ are continuous along every edge $e$, 
but they may in general be discontinuous across inner vertices $v \in \Vi$.
We will further make use of the space
\begin{align*}
H_0(\div) &:= \{ m \in H^1(\E) : \sum\nolimits_{e \in \E(v)} m_e(v) n_e(v) = 0\quad \forall v \in \V \}
\end{align*}
of \emph{conservative flux functions} which vanish at the boundary of the network.
Let us finally note that for all $p,m \in H^1(\E)$ we have the following integration-by-parts formula
\begin{align*} 
(\dx' p,m)_\E 
= -(p, \dx' m)_\E + \sum\nolimits_{v \in \V} \sum\nolimits_{e \in \E(v)} p_e(v) m_e(v) n_e(v).
\end{align*}
This identity follows directly from the definition of $(\dx' p,m)_\E =\sum_{e \in \E} (\dx p_e, m_e)_e$, integration-by-parts on every edge, and summation over all elements. The interface terms drop out if, in addition, $p$ is continuous across junctions $v$ and if $m \in H_0(\div)$; this will be utilized below.
\subsection{Problem description}

Let $\G=(\V,\E,\ell)$ be a geometric graph as introduced in the previous section.
The problem under investigation then consists of the following differential and algebraic equations. 
We look for functions $\rho$, $m$ of time and space, such that 
\begin{align} 
\dt \rho_e + \dx m_e &= 0, && \forall e \in \E  \label{eq:sys1}\\
\dt m_e + \dx \left(\frac{m_e^2}{\rho_e} + p_e\right) &= a_e \rho_e \dx\left( \frac{1}{\rho_e^2} \dx m_e\right) - b_e \frac{|m_e| m_e }{\rho_e}, && \forall e \in \E. \label{eq:sys2}
\end{align}
Recall that $f_e = f|_e$ denotes the restriction of a function $f$ to the edge $e$.
The equations here and below are tacitly assumed to hold for all $t > 0$.
We further assume that the pressure and potential energy density are related to the density of the fluid by 
\begin{align} \label{eq:sys3}
p_e(\rho) = c_e \rho^\gamma 
\qquad \text{and} \qquad 
P_e(\rho) = \frac{c_e}{\gamma-1} \rho^\gamma, \qquad \gamma>1.
\end{align}
More general equations of state can be handled without much difficulty. In our analysis, we will actually only make use of the basic identity \cite{NovotnyStraskraba04}
\begin{align*}
\rho P_e'(\rho) - P_e(\rho) = p_e(\rho).
\end{align*}
This allows to apply our basic arguments directly also to the isothermal flow of real gases.
The parameters $a_e,b_e,c_e$ are assumed to be non-negative and constant on every pipe.The value of $c_e>0$ may be different on every edge $e$, which allows to deal with pipes of different cross-section.
In addition to the conservation laws on the individual pipes, we assume that the following coupling conditions hold at every internal vertex $v \in \Vi$ of the network:
\begin{align}
\sum\nolimits_{e \in \E(v)} m_e(v) n_e(v) &= 0, && \forall v \in \Vi,  \label{eq:sys4}\\
\frac{m_e(v)^2}{2 \rho_e(v)^2}+P_e'(\rho_e(v))  &= h_v + \frac{a_e}{\rho_e^2} \dx m_e(v), && \forall v \in \Vi, \ e \in \E(v). \label{eq:sys5}
\end{align}
To complete the system description, additional conditions have to be imposed at the boundary vertices. For ease of presentation, we assume here that the network is closed such that no mass can enter or leave the network via the boundary, i.e.,
\begin{align} 
m_e(v) n_e(v) &= 0, \qquad \forall v \in \Vb. \label{eq:sys6}
\end{align}
Again, more general boundary conditions could be considered without much difficulty.
To uniquely determine the solution, we finally assume access to the initial values
\begin{align}
\rho(0)=\rho_0 \qquad \text{and} \qquad m(0)=m_0 \label{eq:sys7}
\end{align}
for the density and the mass flux. 
These will only play a minor role in our considerations. 
To ensure that all equations are well-defined, we have to require that 
the functions $\rho$, $m$ have certain smoothness properties and that $\rho > 0$ during the evolution. 
\begin{notation} \label{not:smooth}
By {\em smooth solution} of \eqref{eq:sys1}--\eqref{eq:sys6}, we understand a pair of functions 
\begin{align*}
(\rho,m) \in C^1([0,T];L^\infty(\E) \times L^2(\E)) \cap C([0,T];H_+^1(\E) \times H^2(\E)) 
\end{align*}
satisfying the equations for all $t>0$ and a.e. on $\E$.
Here $H^1_+(\E) = \{\rho \in H^1(\E) : \exists \underline \rho>0 : \rho \ge \underline \rho \mbox{ a.e. on } \E \}$ shall denote the space of uniformly positive piecewise smooth functions. 
\end{notation}

\section{A variational principle} \label{sec:variational}

We now present a variational characterization for solutions of \eqref{eq:sys1}--\eqref{eq:sys6} which is directly related
to the conservation of mass and energy. This weak formulation will be at the center of our considerations and later on serve as the starting point for the numerical approximation. 

\subsection{A variational characterization}

The first equation in the variational principle below  directly follows from the continuity equation. To motivate the somewhat unusual form of the second equation, let us consider the change of kinetic energy, given by
\begin{align*}
\frac{d}{dt} \Big(\frac{m^2}{2\rho} \Big) 
&= \Big(\frac{m}{\rho}\Big) \dt m - \Big(\frac{m^2}{2\rho^2}\Big) \dt \rho 
 = \Big(\frac{1}{\rho} \dt m - \frac{m}{2\rho^2} \dt \rho \Big) \ m.  
\end{align*}
The time derivative of the local kinetic energy is thus obtained by a weighted linear combination of the time derivatives arising in \eqref{eq:sys1} and \eqref{eq:sys2}.
With this in mind, we arrive at the following variational characterization of solutions to the compressible flow problem.

\begin{lemma}[Variational principle] \label{lem:variational} $ $\\
Let $(\rho,m)$ be a smooth solution  of \eqref{eq:sys1}--\eqref{eq:sys6} in the sense of Notation~\ref{not:smooth}. 
Then 
\begin{align*}
\left(\dt \rho,q\right)_\E &= - \left(\dx' m,q\right)_\E,  \\
\left(\frac{1}{\rho} \dt m - \frac{m}{2\rho^2} \dt \rho,v\right)_{\!\!\E}
  &= \left(\frac{m^2}{2\rho^2} + P'(\rho) - \frac{a}{\rho^2}\dx' m, \dx' v\right)_{\!\!\E} 
 \! - \! \left(\frac{m}{2\rho^2} \dx' m + b \frac{|m| m}{\rho^2},v\right)_{\!\!\E},
\end{align*}
for all $q \in L^2$, all $v \in H_0(\div)$, and all $t>0$. Note that the functions $\rho$ and $m$ here depend on $t$, while the test functions $q$ and $v$ are independent of time.
\end{lemma}
\begin{proof}
The first equation follows by multiplying \eqref{eq:sys1} with test functions  $q_e \in L^2(e)$, integration over $e$, and summation over all edges. 
Using \eqref{eq:sys1}, one can see that 
\begin{align*}
-\frac{m_e}{2\rho_e^2} \dt \rho_e = \frac{m_e}{2\rho_e^2} \dx m_e. 
\end{align*}
From equation \eqref{eq:sys2}, we then obtain 
\begin{align*}
\frac{1}{\rho_e} \dt m_e 
&= -\frac{1}{\rho_e}\dx \left(\frac{m_e^2}{\rho_e} + p_e(\rho_e)\right) + a_e \dx \left( \frac{1}{\rho_e^2} \dx m_e\right) - b_e \frac{|m_e| m_e}{\rho_e^2}. 
\end{align*}
The first term on the right hand side of the previous equation can be replaced by 
\begin{align*}
\frac{1}{\rho_e} \dx \left(\frac{m_e^2}{\rho_e}\right) 
&= \dx \left(\frac{m_e^2}{\rho_e^2}\right) + \frac{m_e}{\rho_e^2} \dx m_e. 
\end{align*}
Using the constitutive equation \eqref{eq:sys3}, we obtain for the second term 
\begin{align*}
\frac{1}{\rho_e} \dx p_e(\rho_e) 
&= \dx \left(\frac{p_e(\rho_e)}{\rho_e}\right) + \frac{p_e(\rho_e)}{\rho_e^2} \dx \rho_e 
\\&
= \dx P_e'(\rho_e) - \dx \left(\frac{P_e(\rho_e)}{\rho_e}\right) + \frac{p_e(\rho_e)}{\rho_e^2} \dx \rho_e 
 = \dx P_e'(\rho_e). 
\end{align*}
The last identity follows from the relation between pressure and potential energy density. A combination of the four identities stated above directly leads to 
\begin{align*}
\frac{1}{\rho_e} \dt m_e - \frac{m_e}{2\rho_e^2} \dt \rho_e = -\dx \left(\frac{m_e^2}{2\rho_e^2} + P_e'(\rho_e) - \frac{a_e}{\rho_e^2}\dx m_e\right) - \frac{m_e}{2\rho_e^2} \dx m_e - b_e \frac{|m_e| m_e}{\rho_e^2}. 
\end{align*}
Next we multiply this equation by a test function $v_e$ on every edge $e$, 
integrate over $e$, and then sum up over all edges $e$. 
In compact notation, the resulting identity reads
\begin{align*} 
\left(\frac{1}{\rho} \dt m - \frac{m}{2\rho^2} \dt \rho,v\right)_{\!\!\E} 
= -\left(\dx' \left( \frac{m^2}{2\rho^2} + P'(\rho) - \frac{a}{\rho^2} \dx' m \right), v\right)_{\!\!\E}  
 \! - \! \left(\frac{m}{2\rho^2} \dx' m + b \frac{|m| m}{\rho^2},v\right)_{\!\!\E}.
\end{align*}
We can now use integration-by-parts for the first term on the right hand side. 
Using the continuity condition \eqref{eq:sys5} and assuming $v \in H_0(\div)$, all interface terms in the integration-by-parts formula vanish, and we arrive at the second identity of the variational principle.
\end{proof}

\subsection{Conservation of mass and energy}

As an immediate consequence of the variational characterization of smooth solutions, we 
obtain the following global balance relations.
\begin{lemma}[Balance of mass and energy] \label{lem:identities} $ $\\
Let $(\rho,m)$ be a smooth solution of \eqref{eq:sys1}--\eqref{eq:sys5} in the sense of Notation~\ref{not:smooth}. 
Then 
\begin{align*}
\frac{d}{dt} \int_\E \rho dx = 0 
\end{align*}
and 
\begin{align*}
\frac{d}{dt} \int_\E \frac{m^2}{2\rho} + P(\rho) dx +  \int_\E  \frac{a}{\rho^2} |\dx' m|^2 + b \frac{|m|^3}{\rho^2} dx=0, 
\end{align*}
i.e., the total mass is conserved and the total energy is dissipated at any point in time.
\end{lemma}
\begin{proof}
As shown in the previous lemma, any smooth solution satisfies the above variational principle. Testing the first variational equation with $q=1$ leads to 
\begin{align*}
\frac{d}{dt} \int_\E \rho dx 
&= (\dt \rho,1)_\E 
 = -(\dx' m, 1)_\E
 = -\sum_{v \in \V} \sum_{e \in \E(v)} m_e(v) n_e(v)  = 0.
\end{align*}
Here we used integration-by-parts in the third step, 
and the coupling and boundary conditions stated in \eqref{eq:sys4} and \eqref{eq:sys6} in the last step.
For the second assertion, observe that 
\begin{align*}
\frac{d}{dt} \int_\E \frac{m^2}{2\rho} + P(\rho) dx 
&= \left(\frac{1}{\rho} \dt m - \frac{m}{2\rho^2} \dt \rho,m\right)_{\!\E} + \Big(\dt \rho, P'(\rho)\Big)_{\!\E} = (*). 
\end{align*}
Testing the variational principle with $q=P'(\rho)$ and $v=m$, we further obtain
\begin{align*}
(*) &= \left(\frac{m^2}{2\rho^2} + P'(\rho) - \frac{a}{\rho^2} \dx' m,\dx' m\right)_{\!\!\E} 
- \left(\frac{m}{2\rho^2} \dx' m,m\right)_{\!\!\E} - \left(b \frac{|m|m}{\rho^2},m\right)_{\!\!\E} - \Big(\dx' m, P'(\rho)\Big)_{\E} 
\\&
= - \left(\frac{a}{\rho^2} \dx' m, \dx' m\right)_{\!\!\E} - \left(b \frac{|m|m}{\rho^2},m\right)_{\!\!\E}.
\end{align*}
This shows the energy identity stated in the lemma and completes the proof.
\end{proof}
As a direct consequence of the previous lemma, we obtain the following basic identities.
\begin{lemma}[Conservation of mass and energy] $ $\\
Let $(\rho,m)$ be a smooth solution of \eqref{eq:sys1}--\eqref{eq:sys7} in the sense of Notation~\ref{not:smooth}. 
Then 
\begin{align*}
\int_\E \rho(t) dx = \int_\E \rho_0 dx
\end{align*}
and
\begin{align*}
\int_\E \frac{m(t)^2}{2\rho(t)} + P(\rho(t)) dx + \int_0^t \int_\E a \frac{|\dx' m(s)|^2}{|\rho(s)|^2} + b \frac{|m(s)|^3}{|\rho(s)|^2} dx \; ds 
 = \int_\E \frac{m_0^2}{2\rho_0} + P(\rho_0) dx.
\end{align*}
\end{lemma}
This shows that the total mass is conserved and that the energy is decreased only by dissipation due to viscous forces and friction. Note that the sum of total and dissipated energy is conserved as well, which is why we speak of energy conservation also in the presence of viscous forces and friction. The uniform bounds for the energy and the dissipation terms allows us to prove the existence of solutions to finite dimensional approximations later on.

\subsection{Comments on the variational principle}

Let us briefly put the results of the previous lemmas into perspective: The proof of Lemma~\ref{lem:identities} reveals that the variational principle of Lemma~\ref{lem:variational} is very tightly connected to the global conservation laws for mass and energy. The last step in the proof of Lemma~\ref{lem:variational} additionally shows that the coupling conditions \eqref{eq:sys4}--\eqref{eq:sys5} are the natural ones for this weak formulation and that they guarantee that no mass or energy is generated or annihilated at the junctions. 
Note that the convective terms 
arising on the right hand side of the second variational equation are antisymmetric and therefore disappear when testing with $v=m$. Also the pressure terms arising from the first and second equation in the variational principle annihilate naturally.
These inherent symmetry properties are directly related to the particular form of our variational characterization and almost immediately imply the conservation of mass and energy. 
Finally note that no space derivatives of the density appear in the weak formulation.
This substantially simplifies the design and analysis of appropriate numerical approximations in the following sections.

\section{Galerkin approximation in space} \label{sec:galerkin}

As a first step towards the numerical solution of problem \eqref{eq:sys1}--\eqref{eq:sys7}, we now consider a conforming Galerkin approximation in space by a mixed finite element method. We only discuss the lowest order approximation in detail.
Similar arguments may however also be used for the design and analysis of higher order approximations. 

\subsection{The mesh and polynomial spaces}

Let $[0,\ell_e]$ be the interval related to the edge $e$.  
We denote by $T_h(e) = \{K\}$ a uniform mesh of $e$ with elements $K$ of length $h_K$. 
The global mesh is defined as $T_h(\E) = \{T_h(e) : e \in \E\}$, 
and as usual, we denote by $h=\max_{K \in T_h(\E)} h_K$ the global mesh size. 
Let us define spaces of piecewise polynomials by
\begin{align*}
 P_k(T_h(\E)) &= \{v \in L^2(\E) : v|_e \in P_k(T_h(e)) \ \forall e \in \E\}. 
\end{align*}
Here $P_k(T_h(e)) = \{v \in L^2(e) : v|_K \in P_k(K), \ K \in T_h(e)\}$ is the space of piecewise polynomials on the mesh $T_h(e)$ of a single edge $e$, and $P_k(K)$ 
is the space of polynomials of degree $\le k$ on the subinterval $K$. 
For approximation of the density $\rho$ and the mass flux $m$, we consider functions in the finite element spaces
\begin{align} \label{eq:spaces}
V_h = P_{1}(T_h(\E)) \cap H_0(\div)
\quad \text{and} \quad 
Q_h = P_{0}(T_h(\E)). 
\end{align}
These spaces have good approximation properties and have been used successfully for the numerical approximations of damped wave propagation on networks; see \cite{EggerKugler16} for details.

\subsection{The Galerkin semi-discretization}

For the space discretization of the compressible flow problem \eqref{eq:sys1}--\eqref{eq:sys7},
we consider the following finite dimensional approximations.

\begin{problem}[Semi-discretization] \label{prob:semi} 
Find $(\rho_h,m_h) \in C^1(0,T;Q_h \cap V_h)$ with $\rho_h(0) = \rho_{0,h}$ and $m_h(0)=m_{0,h}$, 
and such that for all $v_h \in V_h$ and $q_h \in Q_h$, and every $t \in [0,T]$, there holds
\begin{align*}
\left(\dt \rho_h,q_h\right)_\E 
  &= - \left(\dx' m_h,q_h\right)_\E,  \\
\left(\frac{1}{\rho_h} \dt m_h - \frac{m_h}{2\rho_h^2} \dt \rho_h,v_h\right)_{\!\!\E} \!
  &= \! \left(\frac{m_h^2}{2\rho_h^2} +P'(\rho_h) - \frac{a}{\rho_h^2} \dx' m_h, \dx' v_h\right)_{\!\!\E} 
 \! - \! \left(\frac{m_h}{2\rho_h^2} \dx' m_h + b \frac{|m_h| m_h}{\rho_h^2},v_h\right)_{\!\!\E}\!.
\end{align*}
Here $0 < \rho_{0,h} \in Q_h$ and $m_{0,h} \in V_h$ are given approximations for the initial values. 
\end{problem}

\subsection{Mass and energy conservation}

We will show below that the semi-discrete problem is uniquely solvable. 
Before doing that, let us present some identities for the conservation of mass and energy, which more or less directly follow from our construction of the variational principle and the fact that we are using a \emph{conforming} Galerkin approximation.

\begin{lemma}[Balance of mass and energy] \label{lem:identitiesh} $ $\\
Let $(\rho_h,m_h)$ be a solution of Problem~\ref{prob:semi}. 
Then 
\begin{align*}
\frac{d}{dt} \int_\E \rho_h dx = 0 
\end{align*}
and 
\begin{align*}
\frac{d}{dt} \int_\E \frac{m_h^2}{2\rho_h} + P(\rho_h) dx 
+ \int_\E \frac{a}{\rho_h^2} |\dx' m_h|^2 + b \frac{|m_h|^3}{\rho_h^2} dx=0, 
\end{align*}
i.e., the mass and energy identities of Lemma~\ref{lem:identities} also hold for the semi-discretization. 
\end{lemma}
\begin{proof}
The assertions follow similar as in the proof of Lemma~\ref{lem:identities}. 
\end{proof}

Upon integration of the two identities, we obtain the following global conservation laws.

\begin{lemma}[Conservation of mass and uniform bounds for the energy] \label{lem:identitiesh2}$ $\\
Let $(\rho_h,m_h)$ be a solution of Problem~\ref{prob:semi}. Then
\begin{align*}
\int_\E \rho_h(t) dx = \int_\E \rho_{0,h} dx
\end{align*}
and
\begin{align*}
\int_\E \frac{m_h(t)^2}{2\rho_h(t)} + P(\rho_h(t)) dx 
+ \int_0^t \int_\E a \frac{|\dx' m_h(s)|^2}{|\rho_h(s)|^2} + b \frac{|m_h(s)|^3}{|\rho_h(s)|^2} dx \; ds 
= \int_\E  \frac{m_{0,h}^2}{2\rho_{0,h}} + P(\rho_{0,h}) dx.
\end{align*}
\end{lemma}
Let us emphasize that these two identities are a verbatim translation of the conservation laws on the continuous level and they are derived in the very same manner.

\subsection{Well-posedness of the semi-discrete scheme}
As a consequence of the uniform bounds for the energy and the dissipation terms
resulting from the previous lemma, we are now able to establish the well-posedness of the semi-discretization.
\begin{lemma}[Well-posedness] \label{lem:wellposedh} 
Assume that $\rho_{0,h} \ge \underline\rho>0$. Then there exists a $T>0$ such that Problem~\ref{prob:semi} has a unique solution $(\rho_h,m_h) \in C^1([0,T);Q_h \times V_h)$ and $\rho_h(t)>0$ for all $0 \le t < T$. If $a \ge \underline a > 0$, 
then we can choose $T=\infty$, i.e., the solution is global.
\end{lemma}
\begin{proof}
Since the spaces $Q_h$ and $V_h$ are finite dimensional and $\rho_h(0)>0$, the local existence of a solution follows readily from the Picard-Lindelöf theorem. Moreover, the solution can be extended uniquely as long as $\rho_h(t)>0$. 
For any element $K$ of the mesh, we have 
\begin{align*}
\frac{d}{dt} \int_K \frac{1}{\rho_h} dx 
&= -\left( \dt \rho_h, \frac{1}{\rho_h^2}\right)_K
 = \left(\dx' m_h, \frac{1}{\rho_h^2}\right)_K \\
&= \left(\frac{1}{\rho_h} \dx' m_h, \frac{1}{\rho_h}\right)_K
 \le \left\|\frac{1}{\rho_h} \dx' m_h\right\|_{L^\infty(K)} \int_K \frac{1}{\rho_h} dx.  
\end{align*}
From the equivalence of norms on finite dimensional spaces, 
we know that 
\begin{align*}
\left\|\frac{1}{\rho_h} \dx' m_h\right\|_{L^\infty(K)} 
\le \left\|\frac{1}{\rho_h} \dx' m_h\right\|_{L^\infty(\E)} 
\le C_h \left\|\frac{1}{\rho_h} \dx' m_h\right\|_{L^2(\E)}.
\end{align*}
By the Gronwall lemma, we thus obtain 
\begin{align*}
\int_K \frac{1}{\rho_h(t)} dx 
&\le e^{\int_0^t C_h \left\|\frac{1}{\rho_h(s)} \dx' m_h(s)\right\|_{L^2(\E)} ds} \int_K \frac{1}{\rho_h(0)} dx.
\end{align*}
The integral in the exponential term can be further estimated by
\begin{align*}
\int_0^t \left\|\frac{1}{\rho_h(s)} \dx' m_h(s)\right\|_{L^2(\E)} ds 
\le \sqrt{t} \Big(\int_0^t \left\|\frac{1}{\rho_h(s)} \dx' m_h(s)\right\|_{L^2(\E)}^2 ds \Big)^{1/2}.
\end{align*}
If $a \ge \underline a > 0$, the term in parenthesis can be bounded uniformly in $t$ by the estimates for the dissipation term provided by the energy identity of Lemma~\ref{lem:identitiesh2}.
Using that $h_{min} \le |K| \le h_{max}$ for some $h_{max},h_{min}>0$ and all elements $K$, 
we then obtain
\begin{align*}
h_{min} \left\|\frac{1}{\rho_h(t)}\right\|_{L^\infty(\E)} 
&\le \max_{K \in T_h(\E)} \int_K \frac{1}{\rho_h(t)} dx 
\le  e^{\sqrt{t} C} h_{max} \left\|\frac{1}{\rho_{0,h}}\right\|_{L^\infty(\E)}.
\end{align*}
This shows that the density stays positive for all time.
From the energy identity, one can then further deduce that also $\|\rho_h\|_{L^\infty(\E)}$ and $\|m_h\|_{L^\infty(\E)}$ at most increase exponentially in time, such that the solution can in fact be extended globally.
\end{proof}

\section{Time discretization}  \label{sec:time}

As a final step in the discretization procedure, 
we now consider the time discretization of the Galerkin approximation considered in the previous section. 
Let $\tau>0$ denote the time-step and set $t^n = n \tau$ for $n \ge 0$. 
Given a sequence $\{d^n\}_{n \ge 0}$, 
we denote by 
\begin{align*}
\dtau d^n :=  \frac{d^n - d^{n-1}}{\tau} \quad \text{for } n \ge 0,
\end{align*}
the backward differences which are taken as approximations for the time derivative terms. 
Different value $\tau_n>0$ for the individual time steps could in principle be chosen as well. 
\subsection{Fully discrete scheme}
For the time discretization of the Galerkin approximation stated in Problem~\ref{prob:semi}, 
we now consider the following implicit time stepping scheme.
\begin{problem}[Fully discrete method] \label{prob:full} 
Let $\rho_{0,h}$ and $m_{0,h}$ be given as in Problem~\ref{prob:semi}, 
and set $\rho_h^0=\rho_{0,h}$ and $m_h^0=m_{0,h}$. 
Then for $n \ge 1$ find $(\rho_h^n,m_h^n)\in Q_h\times V_h$, such that 
\begin{align*}
\left(\dtau \rho_h^n,q_h\right)_\E 
  &= - \left(\dx' m_h^n,q_h\right)_\E,  \\
\left(\frac{1}{\rho_h^{n-1}} \dtau m_h^n - \frac{m_h^n}{2|\rho_h^{n}|^2} \dtau \rho_h^n,v_h\right)_{\!\!\E}
  &= \left(\frac{|m_h^n|^2}{2|\rho_h^n|^2} + P'(\rho_h^n) - \frac{a}{|\rho_h^n|^2} \dx' m_h^n, \dx' v_h\right)_{\!\!\E} 
\\ & \qquad \qquad \qquad 
 \! - \! \left(\frac{m_h^n}{2|\rho_h^n|^2} \dx' m_h^n, v_h\right)_{\!\!\E} 
 \! - \! \left(b \frac{|m_h^n| m_h^n}{|\rho_h^n|^2},v_h\right)_{\!\!\E}
\end{align*}
hold for all test functions $q_h \in Q_h$ and $v_h \in V_h$.
\end{problem}

\noindent
The well-posedness of the method will be stated below. Before doing that, 
let us summarize the basic conservation principles for mass and energy for the full discretization.

\subsection{Conservation of mass and dissipation of energy} 

The particular form of the fully discrete variational problem allows us to immediately obtain the following balance relations for mass and energy by appropriate choice of the test functions.
\begin{lemma}[Conservation of mass and uniform bounds for the energy] $ $ \label{lem:identitieshh}\\
Let $(\rho_h^n,m_h^n)_{n \ge 0}$ be a solution of Problem~\ref{prob:full}. 
Then for all $n \ge 0$ we have 
\begin{align*}
\int_\E \rho_h^n dx = \int_\E \rho_{0,h}
\end{align*}
and 
\begin{align*}
\int_\E \frac{|m_h^n|^2}{2\rho_h^n} + P(\rho_h^n) dx 
+ \tau  \sum_{k=1}^n \int_\E a\frac{|\dx' m_h^k|^2}{|\rho_h^k|^2}  + b \frac{|m_h^k|^3}{|\rho_h^j|^2} dx
\le \int_\E \frac{m_{0,h}^2}{2\rho_{0,h}} + P(\rho_{0,h}) dx.
\end{align*}
The mass is thus conserved and energy is dissipated effectively by viscous forces and friction.
\end{lemma}
The implicit time integration scheme yields some extra numerical dissipation,
which leads to the inequality instead of equality in the energy balance. 
The uniform bounds for energy and dissipation again plays an important role for the well-posedness of the scheme. 
\begin{proof}
The first identity follows by testing Problem~\ref{prob:full} with $q_h=1$ and $v_h=0$.
Now turn to the second: By elementary manipulations, we can split 
\begin{align*}
\int_\E \frac{|m_h^n|^2}{2\rho_h^n} dx 
&= \int_\E \frac{|m_h^{n-1}|^2}{2\rho_h^{n-1}} dx
+ \int_\E \frac{|m_h^n|^2}{2 \rho_h^n} - \frac{|m_h^n|^2}{2 \rho_h^{n-1}} dx 
+ \int_\E \frac{|m_h^n|^2}{2 \rho_h^{n-1}} - \frac{|m_h^{n-1}|^2}{2 \rho_h^{n-1}} dx. 
\end{align*}
Now recall that for any convex differentiable scalar function $f(y)$, we have 
\begin{align*}
f(y^n) \le f(y^{n-1}) + f'(y^{n}) (y^n - y^{n-1}).  
\end{align*}
Applying this to the functions $f(\rho)=\frac{|m_h^n|^2}{2\rho}$ and $f(m)=\frac{m^2}{2 \rho_h^{n-1}}$, which are both convex with respect to their arguments, we obtain 
\begin{align*}
\int_\E \frac{|m_h^n|^2}{2\rho_h^n} dx 
&\le \int_\E \frac{|m_h^{n-1}|^2}{2\rho_h^{n-1}} dx 
 -  \int_\E \frac{|m_h^n|^2}{2 |\rho_h^n|^2} (\rho_h^n - \rho_h^{n-1}) dx 
 +  \int_\E \frac{m_h^n}{\rho_h^{n-1}} (m_h^n - m_h^{n-1}) dx \\
&= \int_\E \frac{|m_h^{n-1}|^2}{2\rho_h^{n-1}} dx 
   + \tau \left(\frac{1}{\rho_h^{n-1}} \dtau m_h^n - \frac{m_h^n}{2 |\rho_h^n|^2} \dtau \rho_h^n, m_h^n\right)_{\!\!\E}. 
\end{align*}
For the last term we can use the discrete variational principle. 
Since the potential energy density is again a convex function, one obtains in a similar manner that
\begin{align*}
\int_\E P(\rho_h^n) dx 
&\le \int_\E P(\rho_h^{n-1}) dx 
 + \tau \left(\dtau \rho_h^n, P'(\rho_h^n)\right)_{\E},
\end{align*}
and the last term can again be treated by the discrete variational principle.
The energy inequality for $n=1$ now follows from the definition of $(\rho_h^n,m_h^n)$ in Problem~\ref{prob:full} with the same reasoning as in Lemma~\ref{lem:identities}.
The result for $n>1$ is obtained by recursion.
\end{proof}

\subsection{Well-posedness of the fully discrete scheme}
Before closing this section, let us establish a basic result concerning the solvability of the nonlinear problems that arise in every single time step of of the fully discrete scheme.
\begin{lemma}[Well-posedness] \label{lem:welposedhh} 
Let $(\rho_h^{n-1},m_h^{n-1}) \in Q_h \times V_h$ with $\rho_h^{n-1} > 0$. Then for $\tau>0$ sufficiently small, the system for determining $(\rho_h^{n},m_h^{n})$ in Problem~\ref{prob:full} has a unique solution with $\rho_h^{n} > 0$. 
If in addition $a \ge \underline a > 0$, then a solution exists for any choice of the time step. 
\end{lemma}
\begin{proof}
The first assertion follows from the implicit function theorem. To show the second assertion, we can proceed with similar arguments as in the proof of Lemma~\ref{lem:wellposedh} to obtain a uniform bound $\|\frac{1}{\rho_h^n}\|_{L^\infty(\E)} \le D_h e^{C_h n} \|\frac{1}{\rho_h^0}\|_{L^\infty(\E)}$ for any solution $(\rho_h^n,m_h^n)$. 
Via the discrete energy estimate of Lemma~\ref{lem:identitieshh}, we then also obtain uniform a-priori bounds for $\|m_h^n\|_{L^2(\E)}$ and $\|p(\rho_h)\|_{L^1(\E)}$, and by equivalence of norms on finite dimensional spaces also for $\|m_h^n\|_{L^\infty(\E)}$ and $\|\rho_h^n\|_{L^\infty(\E)}$. Existence of a solution for the next time step then follows by a homotopy argument and Brouwer's fixed point theorem. 
\end{proof}

\begin{remark} \label{rem:wellposedhh}
The previous lemma guarantees uniqueness of the solution $(\rho_h^n,m_h^n)$ only for a sufficiently small time step $\tau$. The size of the admissible time step will in general depend on $\rho_h^{n-1}$ and $m_h^{n-1}$, in particular on the lower bounds for $\rho_h^{n-1}$. As long as the density $\rho_h^{n-1}$ stays well away from zero, the time step problem will be uniquely solvable for reasonably large time steps, which we also observe in our numerical experiments.
\end{remark}

\section{Solution of the nonlinear problems} \label{sec:implementation}

Before we proceed to numerical tests, let us discuss in some more detail the actual 
implementation of the fully discrete scheme stated in Problem~\ref{prob:full}.
For the solution of the nonlinear system in the $n$th time step, we consider 
the following fixed-point iteration. 
\begin{problem}[Fixed-point iteration for the individual time step] \label{prob:fixed} $ $\\
Set $\widetilde \rho_h = \rho_h^{n-1}$ and $\widetilde m_h = m_h^{n-1}$. 
For $k=1,2,\ldots$ solve for $(\rho_h,m_h)$ the system
\begin{align*}
&\frac{1}{\tau}\left(\rho_h,q_h\right)_\E + \left(m_h,q_h\right)_\E = \frac{1}{\tau}\left(\rho_h^{n-1},q_h\right)_\E, \hspace*{5em} \\
&\frac{1}{\tau}\left(\frac{1}{\rho_h^{n-1}}  m_h ,v_h\right)_{\!\!\E}
- \left(\frac{\widetilde m_h}{2 \widetilde \rho_h^2} m_h + \frac{P'(\widetilde\rho_h)}{\widetilde \rho_h} \rho_h - \frac{a}{|\widetilde \rho_h|^2} \dx' m_h , \dx' v_h\right)_{\!\!\E} \hspace*{5em} \\
&\qquad \qquad \
+  \left(\frac{\widetilde m_h}{2 \widetilde \rho_h^2} \dx' m_h + \frac{b|\widetilde m_h|}{\widetilde \rho_h^2} m_h, v_h\right)_{\!\!\E} 
=\frac{1}{\tau}\left(\frac{1}{\rho_h^{n-1}}  m_h^{n-1} + \frac{\widetilde m_h}{2\widetilde \rho_h^2}  \left(\widetilde \rho_h - \rho_h^{n-1}\right),v_h\right)_{\!\!\E}.
\end{align*}
Then set $\widetilde \rho_h = \rho_h$ and $\widetilde m_h=m_h$, and repeat with $k=k+1$.  
\end{problem}
Note that the systems to be solved in every step of the fixed-point iteration are linear now. 
The iteration yields reasonable approximations for the nonlinear system to be solved.
\begin{lemma}[Equivalence]
Any fixed-point $(\rho_h,m_h)$ of the iteration stated in Problem~\ref{prob:fixed} with $\rho_h>0$ is also a solution $(\rho_h^n,m_h^n)$ of the nonlinear system in Problem~\ref{prob:full}.
\end{lemma}
\begin{proof}
Setting $\rho_h=\widetilde \rho_h=\rho_h^n$ and $m_h=\widetilde m_h=m_h^n$ in the above 
iteration directly leads to the nonlinear system of Problem~\ref{prob:full}.
\end{proof}

Due to the special structure of the linear systems characterizing the fixed-point iteration of Problem~\ref{prob:fixed}, we are able to guarantee that the iteration is well-defined.
\begin{lemma}[Well-posedness]
Let $\rho_h^{n-1}$, $m_h^{n-1}$, $\widetilde \rho_h$, and $\widetilde w_h$ be given and  $\rho_h^{n-1},\widetilde \rho_h>0$. Then there exists a unique solution $(\rho_h,m_h)$ of the linear system in Problem~\ref{prob:fixed}.
\end{lemma}
\begin{proof}
Note that the system for determining $(\rho_h,m_h)$ in iteration $k$ is linear and finite. 
Testing with $q_h=\frac{P'(\widetilde \rho_h)}{\widetilde \rho_h} \rho_h$ and $v_h=m_h$ and summing up the two equations leads to
\begin{align*}
\frac{1}{\tau} \left(\frac{P'(\widetilde \rho_h)}{\widetilde \rho_h} \rho_h,\rho_h \right)_{\!\!\E}
+ \frac{1}{\tau} \left(\frac{1}{\rho_h^n} m_h,m_h\right)_{\!\!\E} + \left(\frac{a}{\widetilde \rho_h^2} \dx' m_h,\dx' m_h\right)_{\!\!\E} + \left(\frac{b|\widetilde m_h|}{\widetilde \rho_h^2} m_h, m_h\right)_{\!\!\E} 
=  
\widetilde C,
\end{align*}
where $\widetilde C$ only depends on the known quantities $\rho_h^{n-1}$, $\widetilde \rho_h$, and $\widetilde m_h$.
Observe that many of the terms dropped out because of antisymmetry of the particular linearization defining the fixed-point iteration.
Since $\rho_h^{n-1}$, $\widetilde \rho_h$ and $P'(\widetilde \rho_h)$ are positive by assumption, the above identity already yields the uniqueness and hence also the existence of the solution.
\end{proof}
\begin{remark}
By a homotopy argument, one can again show that for $\tau>0$ sufficiently small, 
the solution provided by the previous lemma satisfies $\rho_h>0$. 
In order to guarantee the well-posedness for arbitrary step size $\tau$ and all iterates, one can replace the terms $\widetilde \rho_h$ in the denominators of the fixed point iteration by $\max\{\underline \rho,\widetilde \rho_h\}$ for some fixed value $\underline \rho>0$. 
If in addition $a \ge \underline{a}>0$, we actually expect positivity of $\rho_h$ and a unique fixed-point for reasonably large time step; see also Remark~\ref{rem:wellposedhh}.
For the solution of the nonlinear systems of Problem~\ref{prob:full}, 
one may alternatively also utilize a Newton-type iteration. 
\end{remark}

\section{Numerical tests} \label{sec:numerics}

We now complement our theoretical investigations with some computational results. 
For the numerical solution of all test problems, we consider the fully discrete scheme defined in Problem~\ref{prob:full}. For the solution of the nonlinear systems in every time step, we utilize the fixed point iteration given in Problem~\ref{prob:fixed}.
All computations are carried out on uniform meshes with mesh size $h$ and a fixed time step $\tau$. We report on the particular choices below.

\subsection{A shock tube problem}
As a first test problem, we consider the undamped system \eqref{eq:sys1}--\eqref{eq:sys2} with $a=b=0$ on a single pipe $e=[-5,5]$ directed from left to right. For the state equation \eqref{eq:sys3}, we use $c=1/2$ and $\gamma=2$, and as initial conditions, we choose 
\begin{align*}
\rho_0(x)=2-\text{sgn}(x) \qquad \text{and} \qquad m_0(x) = 0.
\end{align*}
This setting amounts to the dam-break test problem for the shallow water equations discussed in \cite[Example~13.4]{LeVeque02}. 
In Figure~\ref{fig:shocktube}, we depict the numerical solutions obtained with our method with mesh size $h=0.01$ and time step size $\tau=0.005$. Two fixed point iterations were used for the solution of the nonlinear problems in every time step.

\begin{figure}[ht!]
\begin{center}
\includegraphics[width=0.41\textwidth]{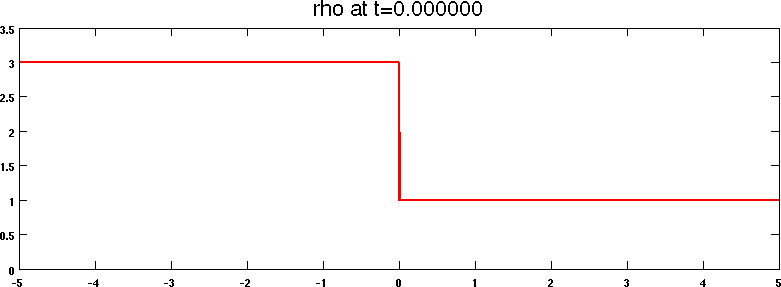}
\hspace*{1em}
\includegraphics[width=0.41\textwidth]{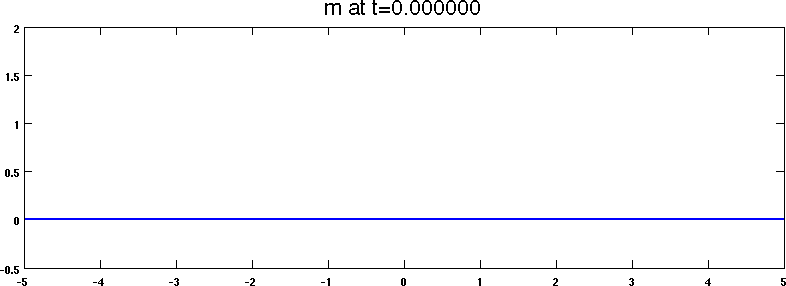} \\[3ex]
\includegraphics[width=0.41\textwidth]{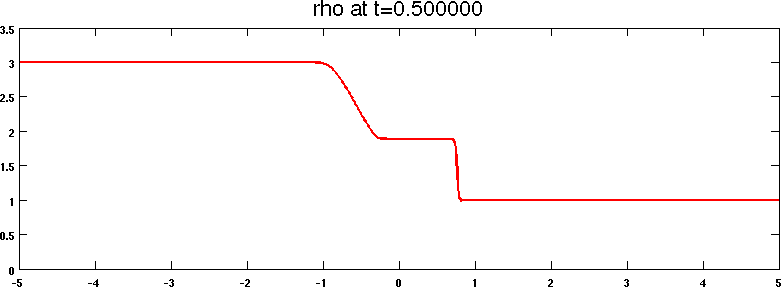}
\hspace*{1em}
\includegraphics[width=0.41\textwidth]{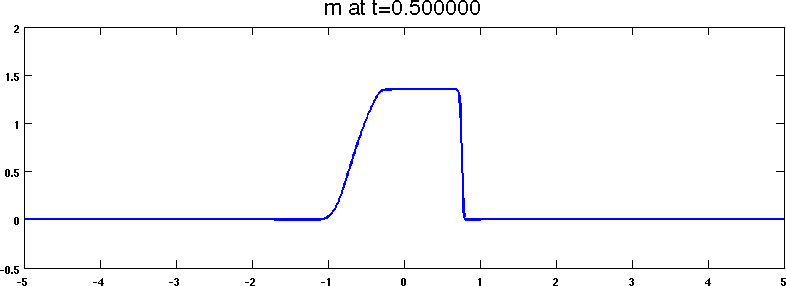} \\[3ex]
\includegraphics[width=0.41\textwidth]{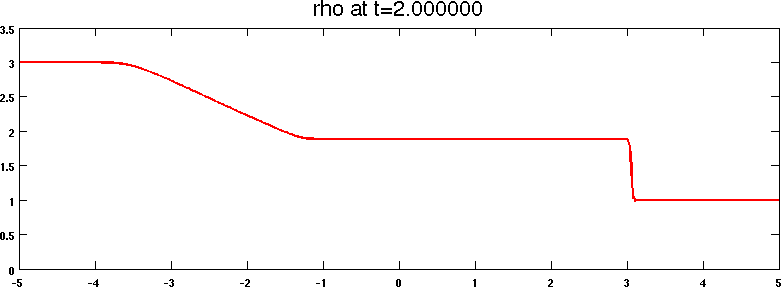}
\hspace*{1em}
\includegraphics[width=0.41\textwidth]{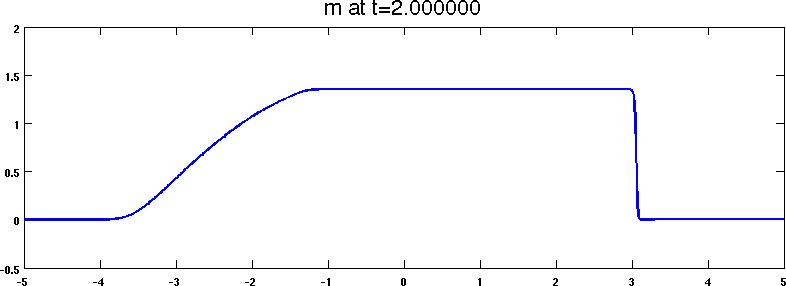}
\end{center}
\caption{Solution for the shock tube problem: A shock wave is propagating to the right and a rarefaction wave to the left. The results are in very good agreement with the ones presented in \cite[Fig.~13.4]{LeVeque02}\label{fig:shocktube}.}
\end{figure}

The solution to the problem here consists of a shock wave propagating to the right and a rarefaction wave propagating to the left. The approximations obtained with our numerical method are in very good agreement with the ones reported in \cite[Fig.~13.4]{LeVeque02}. 
Some slight smoothing of the discontinuity and the kinks can be observed due to the dissipative nature of the implicit time stepping scheme mentioned after Lemma~\ref{lem:identitieshh}. The effect of smoothing slightly increases when the time step $\tau$ is further enlarged. 
The total mass of the system is conserved exactly while the total energy is slightly decreasing over time due to numerical dissipation. For the numerical solution depicted in Figure~\ref{fig:shocktube}, 
the energy at the final time is $E(2.0)=0.983 \times E_0$, which amounts to an energy loss of about 1.7\% due to numerical dissipation. The loss of energy could be further reduced by decreasing the time step size.

\subsection{Convergence towards steady states in the presence of friction}

As a second test scenario, we consider the flow of gas in a single pipe $e=[-5,5]$.
As a model for the propagation, we here use the system \eqref{eq:sys1}--\eqref{eq:sys2} with $a=0$ and $b=100$, which amounts to the inviscid case with large damping typically observed in gas pipelines \cite{BrouwerGasserHerty11,Osiadacz84}. 
The parameters in the pressure law are again set to $c=1/2$ and $\gamma=2$. 
As initial conditions, we here use 
\begin{align*}
\rho_0=11 \qquad \text{and} \qquad m_0=0. 
\end{align*}
The value for the initial density $\rho_0$ is chosen large enough in order to guarantee that the density stays positive for all $t \to \infty$. As boundary conditions, we consider
\begin{align*}
m(-5,t)=m(5,t)=1 \qquad \text{for } t > 0.
\end{align*}
This means that we start to inject gas at time $t=0$ with a constant rate at the left end of the pipe and we extract the same amount of gas on the right end. 
The total mass of the system is therefore conserved for all time.
In Figure~\ref{fig:gas}, we depict the numerical solutions obtained with our method with mesh size $h=0.01$ and time step size $\tau=0.005$. 
Two fixed point iterations are used again for the solution of the nonlinear problems in every time step. 
\begin{figure}[ht!]
\begin{center}
\includegraphics[width=0.41\textwidth]{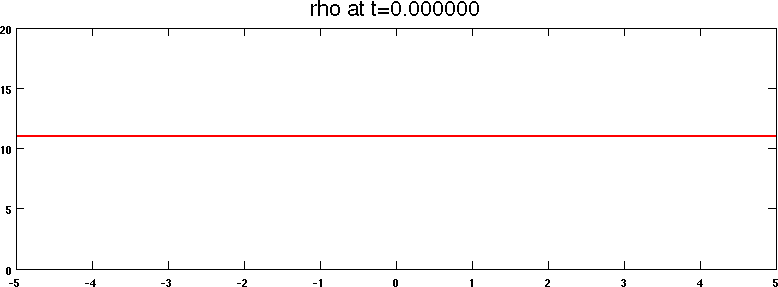}
\hspace*{1em}
\includegraphics[width=0.41\textwidth]{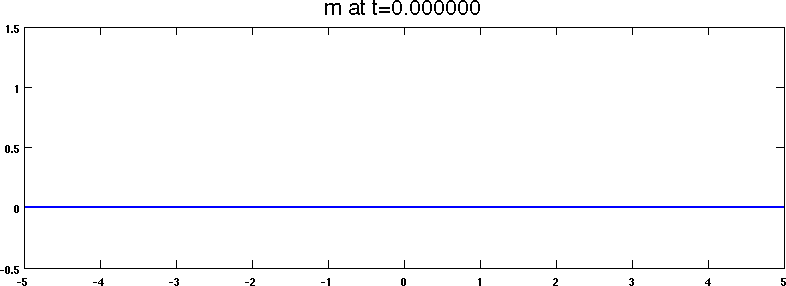} \\[2ex]
\includegraphics[width=0.41\textwidth]{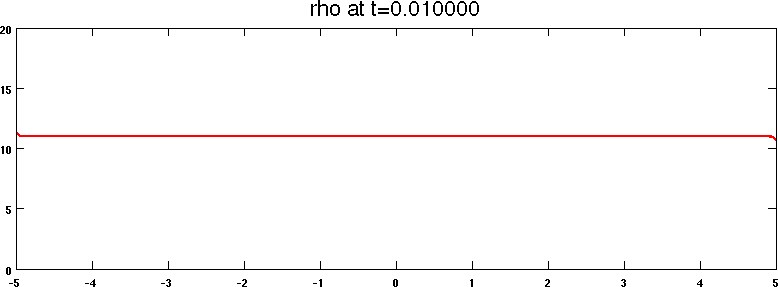}
\hspace*{1em}
\includegraphics[width=0.41\textwidth]{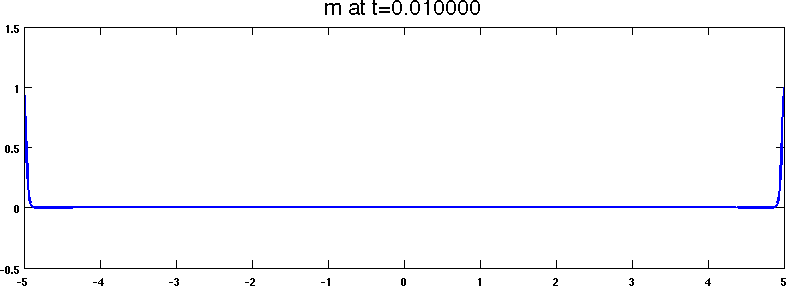} \\[2ex]
\includegraphics[width=0.41\textwidth]{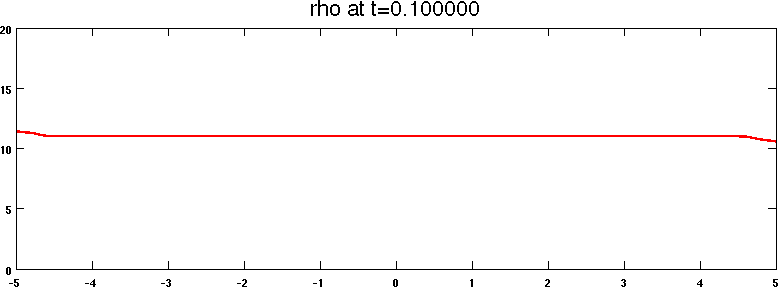}
\hspace*{1em}
\includegraphics[width=0.41\textwidth]{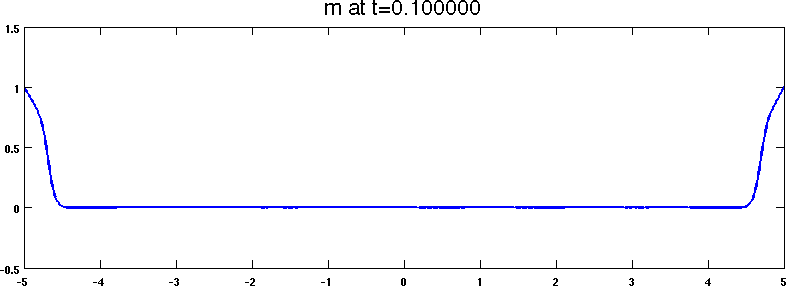} \\[2ex]
\includegraphics[width=0.41\textwidth]{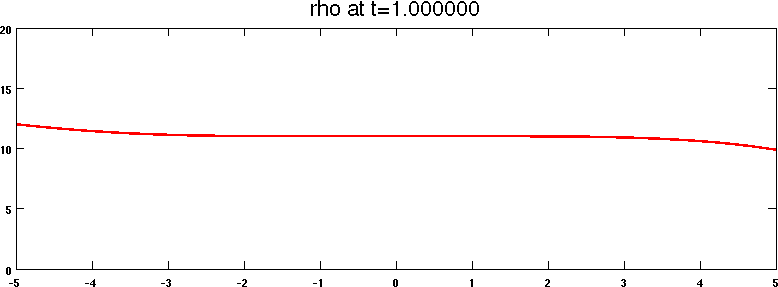}
\hspace*{1em}
\includegraphics[width=0.41\textwidth]{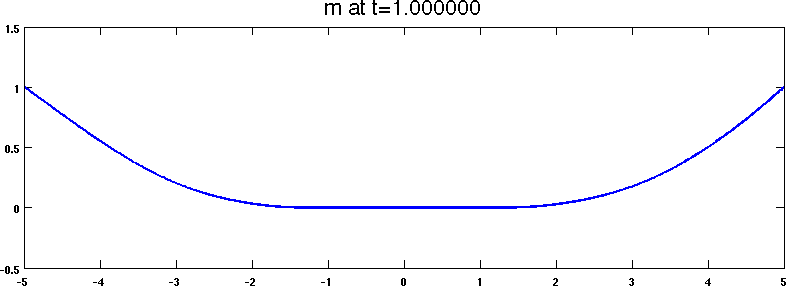} \\[2ex]
\includegraphics[width=0.41\textwidth]{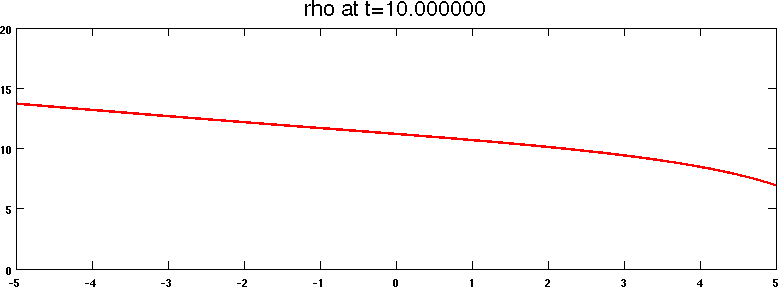}
\hspace*{1em}
\includegraphics[width=0.41\textwidth]{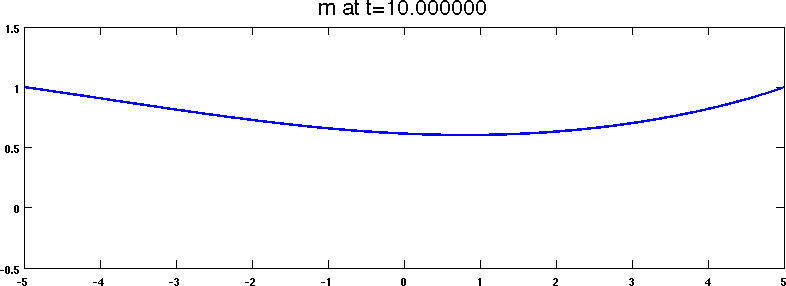} \\[2ex]
\includegraphics[width=0.41\textwidth]{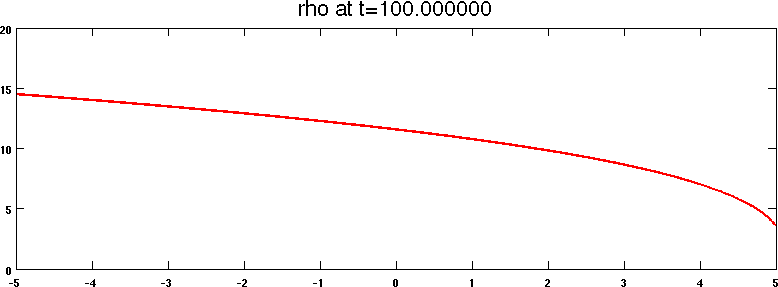}
\hspace*{1em}
\includegraphics[width=0.41\textwidth]{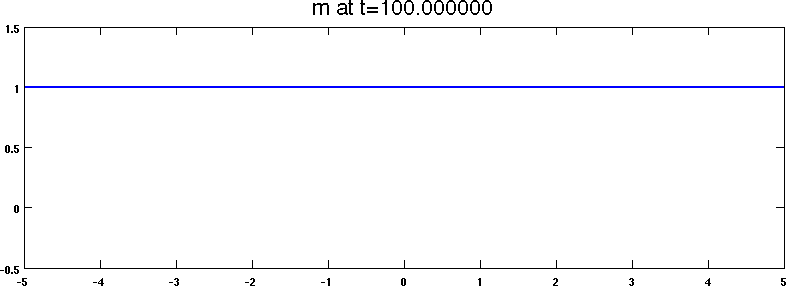}
\end{center}
\caption{Numerical solutions for the gas transport problem: Due to the strong damping, the discontinuities stemming from sudden change of boundary conditions decay rapidly and the system approaches steady state.\label{fig:gas}}
\end{figure}
Due to the sudden change of the boundary conditions, discontinuities are generated at the pipe ends at time $t=0$, but they rapidly decay as an effect of the strong damping. Another consequence of the damping is that the system converges quickly towards a steady state which is the typical situation observed in gas pipelines.  
For the test problem under consideration, the stationary solution is governed by
the system
\begin{align*}
 \bar m &= 1, \\
 \dx \left( \frac{1}{\bar \rho} + \frac{1}{2} \bar \rho^2 \right) &= 100 \frac{1}{\rho} \qquad \text{and} \qquad \int_{-5}^5 \bar \rho dx = 110,
\end{align*}
where we already simplified the momentum equation using the condition $\bar m=1$ and the specific form of the pressure law. 
The last condition comes from the fact that the total mass is conserved. 
The problem of determining $\bar \rho$ using the equations of the second line can be solved numerically by a shooting method.
The approximation for the exact stationary solution obtained in this manner is shown in Figure~\ref{fig:gasstat}. 
\begin{figure}[ht!]
\begin{center}
\includegraphics[width=0.41\textwidth]{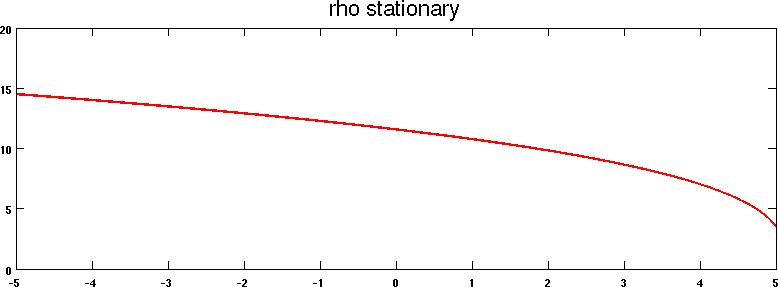}
\hspace*{1em}
\includegraphics[width=0.41\textwidth]{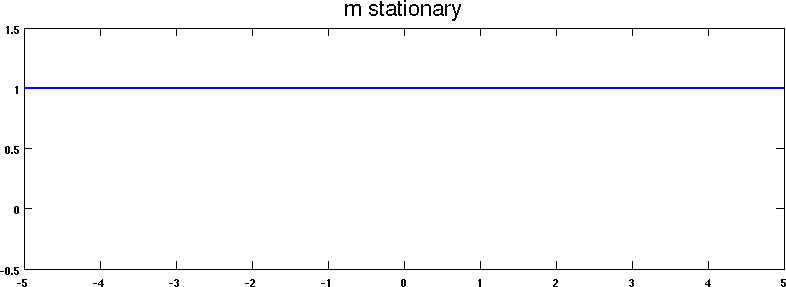} 
\end{center}
\caption{Stationary solution of the gas transport problem computed by the method of Heun and bisection to determine the initial value $\bar \rho(-5)$.\label{fig:gasstat}}
\end{figure}
An inspection of the plots reveals the convergence of the numerical solution towards the correct steady state. Note that the exact conservation of mass on the discrete level is crucial here to get the correct steady state.

\subsection{Gas flow across a junction}

As a last test case, we consider the flow of gas through the simple pipe network depicted in Figure~\ref{fig:graph}. All pipes are chosen to be of unit length. The model parameters are set to $a=0$ and $b=100$, and we take $c=1/2$ and $\gamma=2$ in the pressure law as before. The initial conditions are now chosen as 
\begin{align*}
m_0=0 \qquad \text{and} \qquad \rho_{0,e_1}=5, \ \rho_{0,e_2}=3, \ \rho_{0,e_3}=1.
\end{align*}
These initial values correspond to a specific Riemann problem at the junction \cite{Reigstad15}. 
We further assume that the ports of the network are closed, such that 
\begin{align*}
m(v,t)=0 \qquad \text{for all } v \in \Vb, \ t>0.
\end{align*}
Due to the dissipation of energy caused by friction, we again expect convergence to a
steady flow. As can be verified easily, the stationary solution is here given by
\begin{align*} 
\bar m = 0 \qquad \text{and} \qquad \bar \rho = 3.
\end{align*}
The numerical solutions obtain in our experiments are depicted in Figure~\ref{fig:network}.
\begin{figure}[ht!]
\begin{center}
\includegraphics[width=0.41\textwidth]{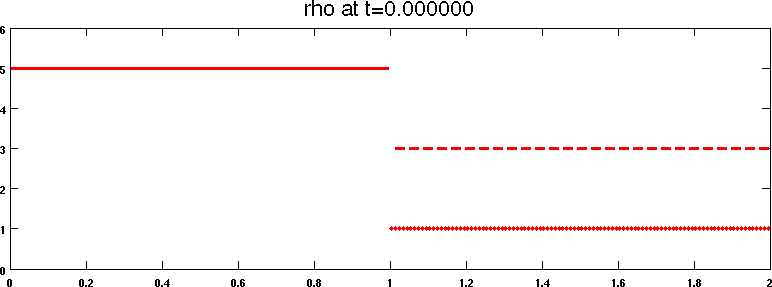}
\hspace*{1em}
\includegraphics[width=0.41\textwidth]{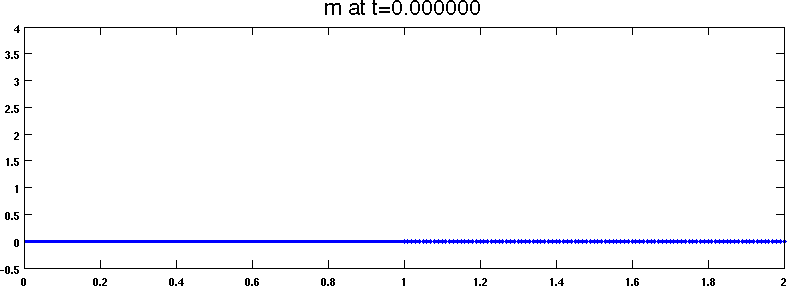} \\[3ex]
\includegraphics[width=0.41\textwidth]{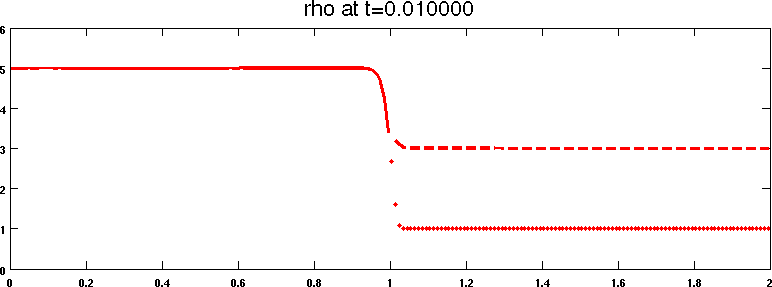}
\hspace*{1em}
\includegraphics[width=0.41\textwidth]{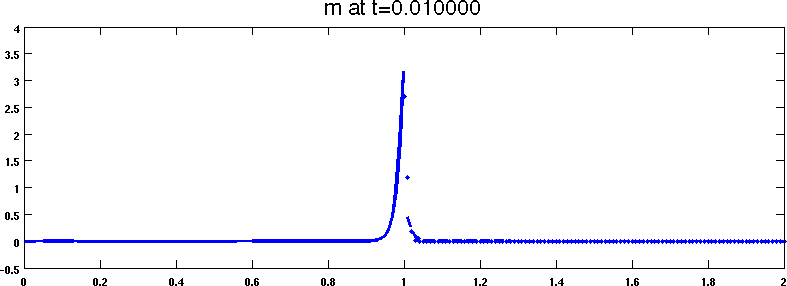} \\[3ex]
\includegraphics[width=0.41\textwidth]{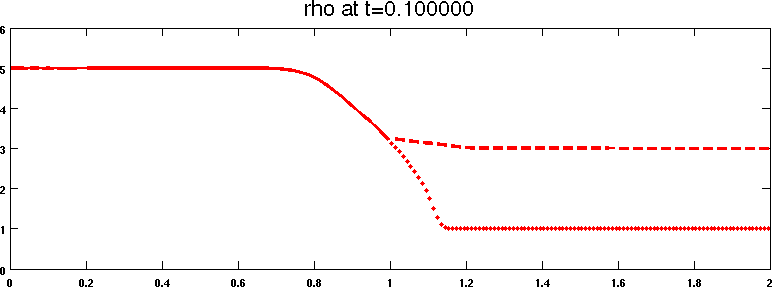}
\hspace*{1em}
\includegraphics[width=0.41\textwidth]{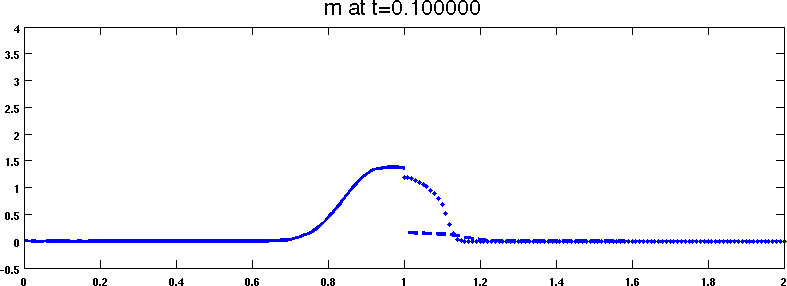} \\[3ex]
\includegraphics[width=0.41\textwidth]{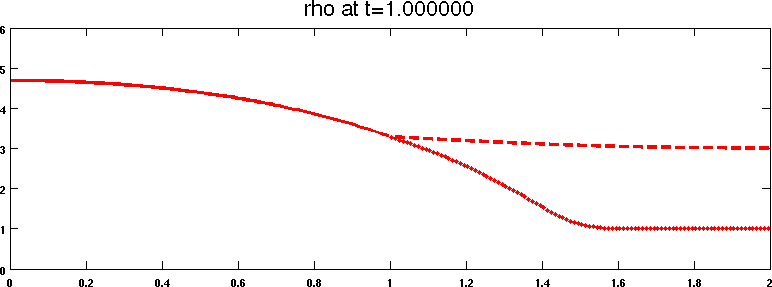}
\hspace*{1em}
\includegraphics[width=0.41\textwidth]{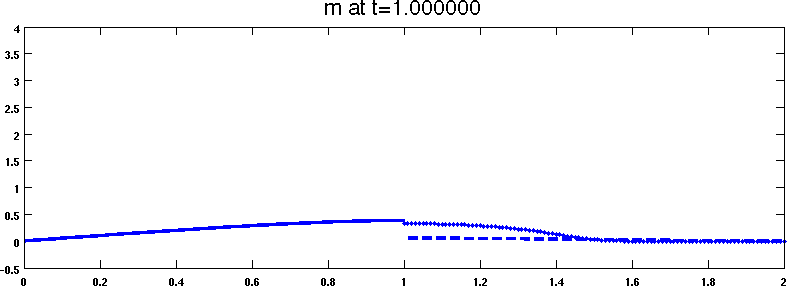} \\[3ex]
\includegraphics[width=0.41\textwidth]{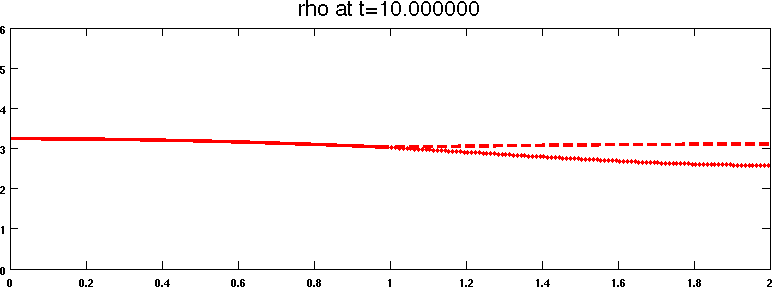}
\hspace*{1em}
\includegraphics[width=0.41\textwidth]{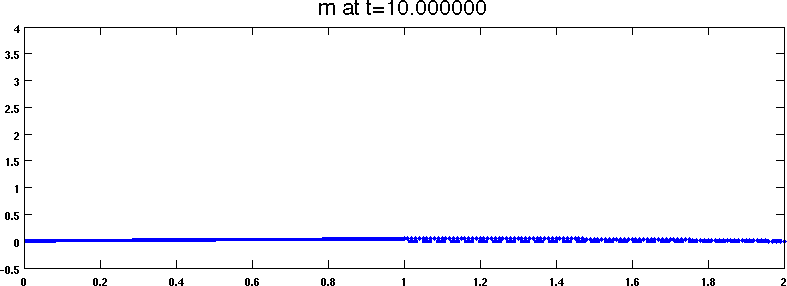} \\[3ex]
\includegraphics[width=0.41\textwidth]{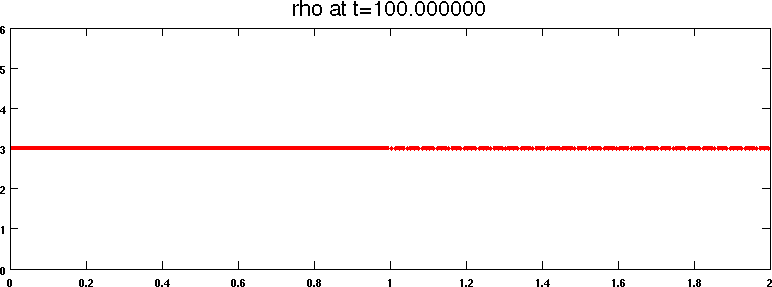}
\hspace*{1em}
\includegraphics[width=0.41\textwidth]{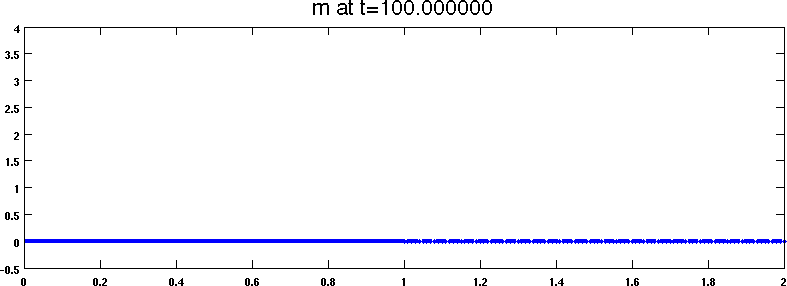}
\end{center}
\caption{Numerical solutions for the gas transport on the network depicted in Figure~\ref{fig:graph}. The x-axis displays the path length along the flow and the junction is located at $x=1$. The different solutions for pipe $1$ to $3$ are depicted with straight, dashed, and dotted lines, respectively.\label{fig:network}}
\end{figure}
The initial value for the density is discontinuous across the junction, but this discontinuity is smeared out quickly due to the damping, and the system approaches a steady state on the long run. The fluxes are discontinuous across the junction for all $t>0$ but as a result of the coupling condition \eqref{eq:sys4}, which is satisfied exactly at every point in time, they always sum up to zero.

\subsection{Further remarks on computational tests}

For the problems under investigation, the numerical diffusion resulting from the implicit time stepping scheme was sufficiently large to yield unique global solutions in all our computations. 
The stability of the scheme was not affected when refining the mesh and convergence was observed for all test cases. 
Some additional artificial viscosity can be added to further regularize the numerical solutions. In fact, very similar results were obtained for all test cases, when the viscosity parameter $a$ was chosen positive and sufficiently small.
Numerical computations were carried out also for other model and discretization parameters. In these tests we observed a very robust behavior of the scheme and the expected convergence with mesh and time step going to zero. 
Our method and analysis is applicable to general finite networks. Computations for more complicated network topologies lead to very similar results and these were therefore not presented here. 

\section{Discussion} \label{sec:discussion}

In this paper, we proposed a novel variational characterization of solutions to compressible flow problems on networks that allowed us to establish conservation of mass and energy in a very direct manner and to apply conforming Galerkin approximations for the discretization in space.
As a particular space discretization, we considered a mixed finite element method.  We proved the well-posedness of the resulting semi-discrete scheme and established the conservation of mass and energy.
We additionally proposed an implicit time stepping leading to a fully discrete scheme that exactly conserves mass and slightly dissipates energy as an effect of numerical dissipation. Well-posedness of the fully discrete scheme was established and a fixed-point iteration for the solution of the nonlinear problems arising in every time step was proposed.
The stability, accuracy, and conservation properties of our numerical method was  demonstrated for some numerical test problems. 
The method proved to be extremely robust concerning the choice of model and discretization parameters, which can be explained by the inherent energy stability of the approximation scheme. 

The energy estimates provided in this paper may serve as a first step towards a complete convergence analysis for the method. A more precise characterization of the numerical dissipation provided by the implicit time stepping scheme may lead to sharper statements about well-posedness of the fully discrete scheme. The extension of our arguments to higher order approximations and more general cases including multi-dimensional problems might be interesting and seems feasible. These topics are left for future research.

\section*{Acknowledgments}
The author would like to thank for support by the German Research Foundation (DFG) via grants IRTG~1529 and TRR~154,
and by the ``Excellence Initiative'' of the German Federal and State Governments via the Graduate School of Computational Engineering GSC~233 at Technische Universität Darmstadt.


\begin{thebibliography}{10}

\bibitem{BandaHertyKlar06a}
M.~K. Banda, M.~Herty, and A.~Klar.
\newblock Coupling conditions for gas networks governed by the isothermal
  {E}uler equations.
\newblock {\em Netw. Heterog. Media}, 1:295--314, 2006.

\bibitem{Berge}
C.~Berge.
\newblock {\em Graphs. 2nd rev.}
\newblock North-Holland, Amsterdam, New~York, Oxford, 1985.

\bibitem{BressanEtAl15}
A.~Bressan, G.~Chen, Q.~Zhang, and S.~Zhu.
\newblock No {BV} bounds for approximate solutions to {$p$}-system with general
  pressure law.
\newblock {\em J. Hyperbolic Differ. Equ.}, 12:799--816, 2015.

\bibitem{BrouwerGasserHerty11}
J.~Brouwer, I.~Gasser, and M.~Herty.
\newblock Gas pipeline models revisited: Model hierarchies, non-isothermal
  models and simulations of networks.
\newblock {\em Multiscale Model. Simul.}, 9:601--623, 2011.

\bibitem{ColomboGaravello06}
R.~M. Colombo and M.~Garavello.
\newblock A well posed {R}iemann problem for the p-system at a junction.
\newblock {\em Netw. Heterog. Media}, 1:495--511, 2006.

\bibitem{ColomboGaravello08}
R.~M. Colombo and M.~Garavello.
\newblock On the {C}auchy problem for the p-system at a junction.
\newblock {\em SIAM J. Math. Anal.}, 39:1456--1471, 2008.

\bibitem{EggerKugler16}
H.~Egger and T.~Kugler.
\newblock Damped wave systems on networks: Exponential stability and uniform
  approximations.
\newblock {\em arXive:1605.03066}, 2016.

\bibitem{Feireisl03}
E.~Feireisl.
\newblock {\em Dynamics of Compressible Flow}.
\newblock Oxford Lecture Series in Mathematics and its Applications. Oxford
  University Press, Oxford, 2003.

\bibitem{GallouetEtAl16}
T.~Gallou{\"e}t, R.~Herbin, D.~Maltese, and A.~Novotny.
\newblock Error estimates for a numerical approximation to the compressible
  barotropic {N}avier-{S}tokes equations.
\newblock {\em IMA J. Numer. Anal.}, 36:543--592, 2016.

\bibitem{Garavello10}
M.~Garavello.
\newblock A review of conservation laws on networks.
\newblock {\em Netw. Heterog. Media}, 5:565--581, 2010.

\bibitem{GugatEtAl12}
M.~Gugat, M.~Herty, A.~Klar, G.~Leugering, and V.~Schleper.
\newblock Well-posedness of networked hyperbolic systems of balance laws.
\newblock In {\em Constrained optimization and optimal control for partial
  differential equations}, volume 160 of {\em Internat. Ser. Numer. Math.},
  pages 123--146. Birkh\"auser/Springer Basel AG, Basel, 2012.

\bibitem{Karper13}
T.~K. Karper.
\newblock A convergent {FEM}-{DG} method for the compressible {N}avier-{S}tokes
  equations.
\newblock {\em Numer. Math.}, 125:441--510, 2013.

\bibitem{Karper14}
T.~K. Karper.
\newblock Convergent finite differences for 1{D} viscous isentropic flow in
  {E}ulerian coordinates.
\newblock {\em Discrete Contin. Dyn. Syst. Ser. S}, 7:993--1023, 2014.

\bibitem{LeVeque02}
R.~J. LeVeque.
\newblock {\em Finite volume methods for hyperbolic problems}.
\newblock Cambridge Texts in Applied Mathematics. Cambridge University Press,
  Cambridge, 2002.

\bibitem{Lions98}
P.-L. Lions.
\newblock {\em Mathematical Topics in Fluid Mechanics}.
\newblock Oxford Lecture Series in Mathematics and its Applications. Clarendon
  Press, Oxford, 1998.

\bibitem{MorinReigstad15}
A.~Morin and G.~A. Reigstad.
\newblock Pipe networks: coupling constants in a junction for the isentropic
  {E}uler equations.
\newblock {\em Energy Procedia}, 64:140--149, 2015.

\bibitem{Mugnolo14}
D.~Mugnolo.
\newblock {\em Semigroup methods for evolution equations on networks}.
\newblock Springer, Cham, 2014.

\bibitem{NovotnyStraskraba04}
A.~Novotny and I.~Straksraba.
\newblock {\em Introduction to the Mathematical theory of compressible flow}.
\newblock Oxford Lecture Series in Mathematics and its Applications. Oxford
  University Press, Oxford, 2003.

\bibitem{Osiadacz84}
A.~Osiadacz.
\newblock Simulation of transient gas flows in networks.
\newblock {\em Int. J. Numer. Meth. Fluids}, 4:13--24, 1984.

\bibitem{ColomboHertySachers08}
M.~H. R.~M.~Colombo and V.~Sachers.
\newblock On 2x2 conservation laws at a junction.
\newblock {\em SIAM J. Math. Anal.}, 40:605--622, 2008.

\bibitem{Reigstad15}
G.~A. Reigstad.
\newblock Existence and uniqueness of solutions to the generalized {R}iemann
  problem for isentropic flow.
\newblock {\em SIAM J. Appl. Math.}, 75:679--702, 2015.

\bibitem{ZarnowskiHoff91}
R.~Zarnowski and D.~Hoff.
\newblock A finite-difference scheme for the {N}avier-{S}tokes equations of
  one-dimensional, isentropic, compressible flow.
\newblock {\em SIAM J. Numer. Anal.}, 28:78--112, 1991.

\bibitem{ZhaoHoff94}
J.~Zhao and D.~Hoff.
\newblock A convergent finite-difference scheme for the {N}avier-{S}tokes
  equations of one-dimensional, nonisentropic, compressible flow.
\newblock {\em SIAM J. Numer. Anal.}, 31:1289--1311, 1994.

\bibitem{ZhaoHoff97}
J.~Zhao and D.~Hoff.
\newblock Convergence and error bound analysis of a finite-difference scheme
  for the one-dimensional {N}avier-{S}tokes equations.
\newblock In {\em Nonlinear evolutionary partial differential equations
  (Beijing, 1993)}, volume~3 of {\em AMS/IP Stud. Adv. Math.}, pages 625--631,
  Providence, RI, 1997. AMS.

\end{thebibliography}

\end{document}